\documentclass[11pt, a4paper]{amsart}

\usepackage{amsfonts,amsmath,amssymb, amscd}
\usepackage{txfonts, mathabx}

\newtheorem{theorem}{Theorem}[section]
\newtheorem{lemma}[theorem]{Lemma}

\newtheorem{prop}[theorem]{Proposition}

\theoremstyle{definition}

\newtheorem{rems}[theorem]{Remarks}

\newcommand\ZZ{\mathbb Z}
\newcommand\FF{\mathbb{F}}

\numberwithin{equation}{section}

\title{The Fourier expansion of $\eta(z)\eta(2z)\eta(3z)/\eta(6z)$}

\author{Christian Kassel}
\address{Christian Kassel: 
Institut de Recherche Math\'e\-ma\-tique Avanc\'ee,
CNRS \& Universit\'e de Strasbourg,
7 rue Ren\'{e} Descartes, 67084 Strasbourg, France}
\email{kassel@math.unistra.fr}
\urladdr{www-irma.u-strasbg.fr/\raise-2pt\hbox{\~{}}kassel/}

\author{Christophe Reutenauer}
\address{Christophe Reutenauer:
Math\'ematiques, Universit\'e du Qu\'ebec \`a Montr\'eal,
Montr\'eal, CP 8888, succ.\ Centre Ville, Canada H3C 3P8}
\email{reutenauer.christophe@uqam.ca}
\urladdr{www.lacim.uqam.ca/\raise-2pt\hbox{\~{}}christo/}



\begin{document}

\begin{abstract}
We compute the Fourier coefficients of the weight one modular form
$\eta(z)\eta(2z)\eta(3z)/\eta(6z)$ in terms of the number of representations of an integer
as a sum of two squares. 
We deduce a relation between this modular form and translates of the modular form
$\eta(z)^4/\eta(2z)^2$.
In the last section we use our main result to give an elementary proof of an identity by Victor Kac.
\end{abstract}

\maketitle

\section{Introduction}

In this note we consider the $\eta$-product 
\begin{equation}\label{def-eta6}
\frac{\eta(z)\eta(2z)\eta(3z)}{\eta(6z)} = \prod_{n\geq 1}\, \frac{(1-q^n)^2}{1 - q^n +q^{2n}}.
\end{equation}
where $q = e^{2\pi i z}$. Recall that $\eta(z)$ is Dedekind's eta function
\begin{equation*}
\eta(z) = e^{\pi i z/12} \, \prod_{n\geq1} \, (1-q^n) .
\end{equation*}
The $\eta$-product $\eta(z)\eta(2z)\eta(3z)/\eta(6z)$ is a modular form of weight~$1$ and level~$6$.
Since it is invariant under the transformation $z \mapsto z+1$, it has a Fourier expansion of the form
\begin{equation}\label{eqFourier}
\frac{\eta(z)\eta(2z)\eta(3z)}{\eta(6z)} = \sum_{n\geq 0} \, a_6(n) \, q^n,
\end{equation}
where the Fourier coefficients $a_6(n)$ are integers. 
For general information on $\eta$-products, see\,\cite[Sect.\,2.1]{Ko}.

Our main result expresses~$a_6(n)$ in terms of the number $r(n)$ of representations
of $n$ as the sum of two squares, i.e the number of elements $(x,y)\in \ZZ^2$ such that $x^2 + y^2 = n$.
Observe that $r(n)$ is divisible by~$4$ for all $n\geq 1$ (for $n=0$ we have $r(0) = 1$).
The sequence~$r(n)$ appears as Sequence~A004018 in\,\cite{OEIS}.

\begin{theorem}\label{th-a6r}
For all non-negative integers $m$ we have
\begin{eqnarray*}
a_6(3m) & = & (-1)^m \, r(3m), \\
a_6(3m+1) & = & (-1)^{m+1} \, \frac{r(3m+1)}{4}, \\
a_6(3m+2) & = & (-1)^{m+1} \, \frac{r(3m+2)}{2}.
\end{eqnarray*}
\end{theorem}

We next relate $\eta(z)\eta(2z)\eta(3z)/\eta(6z)$ to the weight one modular form $\eta(z)^4/\eta(2z)^2$
and two of its translates. 

\begin{theorem}\label{th-linear}
Set $j = e^{2\pi i/3}$.
We have the following linear relation between weight one modular forms:
\begin{equation*}
\frac{\eta(z)\eta(2z)\eta(3z)}{\eta(6z)}
= \frac{1}{4} \frac{\eta(z)^4}{\eta(2z)^2} + \frac{1-j}{4} \frac{\eta(z+1/3)^4}{\eta(2z+2/3)^2} 
+ \frac{1-j^2}{4} \frac{\eta(z+2/3)^4}{\eta(2z+4/3)^2} .
\end{equation*}
\end{theorem}

Both modular forms $\eta(z)\eta(2z)\eta(3z)/\eta(6z)$ and $\eta(z)^4/\eta(2z)^2$ came up naturally
in\,\cite{KR}, where we computed the number~$C_n(q)$ of ideals of codimension~$n$ of 
the algebra $\FF_q[x,y,x^{-1}, y^{-1}]$ of Laurent polynomials in two variables over a finite field~$\FF_q$ of cardinality~$q$.
Equivalently, $C_n(q)$ is the number of $\FF_q$-points of the Hilbert scheme of $n$~points on a two-dimensional torus.
We proved that $C_n(q)$ is the value at $q$ of a palindromic one-variable polynomial $C_n(x) \in \ZZ[x]$
with integer coefficients, which we computed completely (see\,\cite[Th.\,1.3]{KR}).

We also showed (see\,\cite[Cor.\,6.2]{KR})
that the generating function of the polynomials~$C_n(x)$ can be expressed as the following infinite product:
\begin{equation}\label{inf-prod}
1 + \sum_{n\geq 1} \, \frac{C_n(x)}{x^n}  \, q^n
= \prod_{n\geq 1}\, \frac{(1-q^n)^2}{1-(x+ x^{-1})q^n +  q^{2n}} \, .
\end{equation}

It follows from the previous equality that $C_n(1) = 0$. Actually, we proved (see\,\cite[Th.\,1.3 and\,1.4]{KR})
that there exists a polynomial $P_n(x) \in \ZZ[x]$ such that $C_n(x) = (x-1)^2 P_n(x)$. Moreover,
$P_n(x)$ is palindromic, has non-negative coefficients and its value at $x = 1$ is equal to the sum of divisors of~$n$:
$P_n(1) = \sum_{d|n}\, d$.

When $x = e^{2i\pi/k}$ with $k = 2, 3, 4$, or~$6$, then $x+ x^{-1} = 2 \cos(2\pi/k)$ is an integer.
For such an integer~$k$, we define the sequence $a_k(n)$ by
\begin{equation}\label{eq-def}
\sum_{n\geq 0} \, a_k(n) \, q^n =
\prod_{n\geq 1}\, \frac{(1-q^n)^2}{1- 2 \cos(2\pi/k) \, q^n +  q^{2n}} .
\end{equation}
Since $2 \cos(2\pi/k)$ is an integer, so is each~$a_k(n)$. 
It follows from\,\eqref{inf-prod} that these integers are related to the polynomials~$C_n(x)$ by
\begin{equation*}
C_n(e^{2i\pi/k}) = a_k(n) \, e^{2ni\pi/k} .
\end{equation*}
In\,\cite{KR} we computed $a_2(n)$, $a_3(n)$, and $a_4(n)$ explicitly in terms of well-known arithmetical functions.
In particular, we established the equality
\begin{equation}\label{eq-a2r}
a_2(n) = (-1)^n \, r(n),
\end{equation}
where $r(n)$ is the number of representations of $n$ as the sum of two squares.

We also observed in\,\cite[(1.8)]{KR} that 
\begin{equation}\label{eq-a6eta}
\sum_{n\geq 0} \, a_2(n) \, q^n = \frac{\eta(z)^4}{\eta(2z)^2}
\quad\text{and}\quad
\sum_{n\geq 0} \, a_6(n) \, q^n = \frac{\eta(z)\eta(2z)\eta(3z)}{\eta(6z)}.
\end{equation}

The question of finding an explicit expression for~$a_6(n)$ had been left open in\,\cite{KR}. 
This is now solved with Theorem\,\ref{th-a6r} of this note.
In view of this theorem, of\,\eqref{eq-a2r}, and of\,\eqref{eq-a6eta}, 
for all $m\geq 0$ we obtain
\begin{eqnarray}\label{eq-a2a6}
\begin{cases}
\hskip 19pt a_6(3m) = & a_2(3m), \\
\noalign{\medskip}
a_6(3m+1) = & \displaystyle{\frac{a_2(3m+1)}{4}}, \\
\noalign{\medskip}
a_6(3m+2) = & - \displaystyle{\frac{a_2(3m+2)}{2}}.
\end{cases}
\end{eqnarray}

We had experimentally observed (see\,\cite[Footnote\,7]{KR}) that $a_6(n) = 0$ whenever $a_2(n) = 0$.
As a consequence of~\eqref{eq-a2a6} we can now state that $a_6(n) = 0$ if and only if $a_2(n) = 0$, i.e.
if and only $n$ is not the sum of two squares.

\begin{rems}
(a) The sequence~$a_6(n)$ is Sequence~A258210 in\,\cite{OEIS}.
The sequence $a_6(3n+1)$ is probably the opposite of Sequence~A258277 in \emph{loc.\ cit.}

(b) It can be seen from Table\,\ref{table-6th-root} that $a_6(n)$ is not a multiplicative function.
Indeed, $a_6(10) \neq  a_6(2) a_6(5)$ or $a_6(18) \neq  a_6(2) a_6(9)$ or $a_6(20) \neq  a_6(4) a_6(5)$.
\end{rems}

{\tiny
\begin{table}[ht]
\caption{\emph{First values of $a_6(n)$}}\label{table-6th-root}
\renewcommand\arraystretch{1.2}
\noindent\[
\begin{array}{|c||c|c|c|c|c|c|c|c|c|c|c|c|c|c|c|c|c|c|c|c|}
\hline
n & 1 & 2 & 3 & 4 & 5 & 6 & 7& 8& 9 & 10 & 11 & 12 & 13 & 14& 15 & 16 & 17 & 18 & 19 & 20 \\
\hline
\hline
a_6(n) & -1 & -2 & 0 & 1 & 4 & 0 & 0 & -2 & -4 & 2 & 0 & 0 & -2 & 0 & 0 & 1 & 4 & 4 & 0 & -4\\
\hline
\end{array}
\]
\end{table}
}

Theorems\,\ref{th-a6r} and\,\ref{th-linear} will be proved in the next two sections.
In Section\,\ref{sec-Kac} we explain how to obtain an elementary proof of an identity which Victor Kac\,\cite{Ka} obtained
using his theory of contragredient Lie superalgebras.

\section{Proof of Theorem\,\ref{th-a6r}}\label{sec-th-a6r}

\subsection{}\label{ssec-xi}

For any odd integer $m$ we set $\xi(m) = -2 \sin(m\pi/6)$. Because of the well-known properties of the sine function,
$\xi(m)$ depends only on the class of~$m$ modulo~$12$ and we have the following equalities for all odd~$m$:
\begin{equation}\label{eq-sinus}
\xi(-m) = -\xi(m) \quad\text{and}\quad
\xi(m+6) = - \xi(m),
\end{equation}
which is equivalent to $\xi(-m) = -\xi(m)$ and $\xi(6-m) =  \xi(m)$.

We have
\begin{equation}\label{eq-xi}
\xi(m) = 
\begin{cases}
-1 \quad \text{if} \; m \equiv 1 \; \text{or}\;  5 \pmod{12}, \\
-2 \quad \text{if} \; m \equiv 3 \pmod{12}, \\
1 \qquad \text{if} \; m \equiv 7 \; \text{or}\; 11 \pmod{12}, \\
2 \qquad \text{if} \; m \equiv 9 \pmod{12}.
\end{cases}
\end{equation}

Next consider the excess function $E_1(n;4)$ defined by
\begin{equation*}
E_1(n;4) = \sum_{d|n \, , \; d \equiv 1 \hskip -7pt\pmod{4}} \, 1 - \sum_{d|n \, , \; d \equiv -1 \hskip -7pt\pmod{4}} \, 1.
\end{equation*}
It is a multiplicative function, i.e. $E_1(mn;4) = E_1(m;4) \, E_1(n;4)$ whenever $m$ and $n$ are coprime.
It is well known that
the excess function can be computed in terms of the prime decomposition of~$n$.
Write $n = 2^c p_1^{a_1} p_2^{a_2} \cdots q_1^{b_1}q_2^{b_2}\cdots$, where all $p_i$, $q_i$ are distinct prime numbers
such that $p_i \equiv 1 \hskip -3pt\pmod{4}$ and $q_i \equiv 3 \hskip -3pt\pmod{4}$ for all $i\geq 1$. 
Then $E_1(n;4) = 0$ if and only if one of the exponents $b_i$ is odd. If all $b_i$ are even, then
\begin{equation}\label{eq-calculE}
E_1(n;4) = (1 + a_1)(1 + a_2)\cdots .
\end{equation}

In the sequel we will need the following result.

\begin{lemma}\label{lem-xi}
Let $n$ be a positive integer which is not divisible by~$3$.
We have
\begin{equation*}
\sum_{d|n \, , \; d\, \mathrm{odd}} \, \xi(d) = -E_1(n;4) 
\quad\text{and}\quad
\sum_{d|n \, , \; d\, \mathrm{odd}} \, \xi(3d) = -2E_1(n;4) .
\end{equation*}
\end{lemma}

\begin{proof}
Let $d$ be an odd divisor of~$n$; it is not divisible by~$3$ since $n$ is not. Therefore,
$d \equiv 1,5,7$ or $11 \hskip -3pt\pmod{12}$. 
Observe that $d \equiv 1$ or $5 \hskip -3pt \pmod{12}$ if and only if $d \equiv 1 \hskip -3pt\pmod{4}$ 
since $d \equiv 3 \hskip -3pt\pmod{12}$ is excluded.
Similarly, $d \equiv 7$ or $11  \hskip -3pt\pmod{12}$ if and only if $d \equiv 3 \hskip -3pt\pmod{4}$.
Now, $\xi(d) = -1$ if $d \equiv 1$ or $5$, and $\xi(d) = 1$ if $d \equiv 7$ or $11  \hskip -3pt\pmod{12}$.
Consequently, 
\begin{eqnarray*}
\sum_{d|n \, , \; d\, \mathrm{odd}} \, \xi(d)
& = &  \sum_{d|n \, , \; d \equiv 3 \hskip -7pt\pmod{4}} \, 1 - \sum_{d|n \, , \; d \equiv 1 \hskip -7pt\pmod{4}} \, 1 
=  - E_1(n;4) .
\end{eqnarray*}

Similarly, $\xi(3d) = \xi(3) = -2$ if $d \equiv 1$ or $5$, and $\xi(3d) = \xi(9) = 2$ if $d \equiv 7$ or~$11 \hskip-4pt \pmod{12}$.
Therefore,
\begin{eqnarray*}
\sum_{d|n \, , \; d\, \mathrm{odd}} \, \xi(3d)
& = &  \sum_{d|n \, , \; d \equiv 3 \hskip -7pt\pmod{4}} \, 2 - \sum_{d|n \, , \; d \equiv 1 \hskip -7pt\pmod{4}} \, -2 
=   -2 E_1(n;4) .
\end{eqnarray*}
\end{proof}

\subsection{}\label{ssec-prop-a6}

We now express $a_6(n)$ in terms of the function~$\xi$ introduced above.

\begin{prop}\label{prop-a6}
We have
\begin{equation}\label{eq-a6}
a_6(n) = \sum_{d|n \, , \; d\, \mathrm{odd}} \, \xi\left( \frac{2n}{d} - d\right).
\end{equation}
\end{prop}

Note that $2n/d - d$ is an odd integer since $d$ is an odd divisor of~$n$. 

\begin{proof}
Set $u = \pi/k$ and $\omega = d$ in Formula\,(9.3) of~\cite[p.\,10]{Fi}.
It becomes
\begin{equation}\label{eq-Fine}
\sum_{n\geq 0} \, a_k(n) \, q^n = 
1 - 4 \sin(\pi/k) \, \sum_{n\geq 1} \, \left( \sum_{d|n \, , \; d\, \text{odd}} \,  
\sin\left(  \left( \frac{2n}{d} - d\right) \frac{\pi}{k} \right)\right) q^n .
\end{equation}

Consider the special case $k=6$ of\,\eqref{eq-Fine}. 
Since $\sin(\pi/6) = 1/2$, Equality\,\eqref{eq-Fine} becomes
\begin{eqnarray*}
\sum_{n\geq 0} \, a_6(n) \, q^n 
& = & 1 - 2 \, \sum_{n\geq 1} \, \left( \sum_{d|n \, , \; d\, \text{odd}} \,  
\sin\left(  \left( \frac{2n}{d} - d\right) \frac{\pi}{6} \right)\right) q^n \\
& = & 1 + \sum_{n\geq 1} \, \left( \sum_{d|n \, , \; d\, \text{odd}} \,  
\xi\left(  \frac{2n}{d} - d\right) \right) q^n 
\end{eqnarray*}
in view of the definition of~$\xi$.
The formula for~$a_6(n)$ follows.
\end{proof}

\begin{proof}[Proof of Theorem\,\ref{th-a6r}]
Let us first mention the following well-known fact (see\,\cite[\S\,51, Th.\,65]{Di}): 
the number $r(n)$ of representations of~$n$ as a sum of two squares is
related to the excess function~$E_1(n;4)$ by
\begin{equation}\label{eq-r=E}
r(n) = 4 \, E_1(n;4)
\end{equation}
for all $n\geq 0$.
It follows from this fact and from\,\eqref{eq-a2r} that
\begin{equation}\label{eq-a6=E}
a_2(n) = (-1)^n 4 \, E_1(n;4).
\end{equation}
We now distinguish three cases according to the residue of~$n$ modulo~$3$.

(a) We start with the case $n \equiv 1 \pmod{3}$. We have $n = 3\ell +1$ for some non-negative integer~$\ell$.
Since the odd divisors~$d$ of~$n$ are not divisible by~$3$, they must satisfy $d \equiv 1,5,7$ or $11 \hskip -3pt\pmod{12}$. 
Such divisors are invertible $\hskip -3pt\pmod{12}$ and we have $d^2 \equiv 1 \hskip -3pt\pmod{12}$.
Consequently, 
\begin{equation*}
\frac{2n}{d} -d \equiv  \frac{2nd^2}{d} -d \equiv  2nd - d \pmod{12} .
\end{equation*}
Hence,
\begin{equation*}
\xi\left( \frac{2n}{d} -d \right) =  \xi(2nd - d) = \xi(6d\ell + d) = ((-1)^d)^{\ell} \xi(d) = (-1)^{\ell} \xi(d)
\end{equation*}
in view of\,\eqref{eq-sinus}. Therefore, by Proposition\,\ref{prop-a6},
\begin{equation*}
a_6(n) = (-1)^{\ell}  \sum_{d|n \, , \; d\, \mathrm{odd}} \, \xi(d) .
\end{equation*}
Together with Lemma\,\ref{lem-xi} and\,\eqref{eq-a6=E}, this implies
\begin{eqnarray*}
a_6(n) & = &  (-1)^{\ell+1} \, E_1(n;4) = (-1)^{n+ \ell+1} \, a_2(n) /4.
\end{eqnarray*}
Finally observe that $n$ is odd (resp.\ even) if $\ell$ is even (resp.\ odd).
Therefore, $a_6(n) = a_2(n) /4$.

(b) Now consider the case $n \equiv 2 \pmod{3}$. We have $n = 3\ell +2$ for some non-negative integer~$\ell$.
Again the odd divisors~$d$ of~$n$ must satisfy $d \equiv 1,5,7$ or $11 \pmod{12}$
since they are not divisible by~$3$. 
Consequently, as above, 
\begin{equation*}
\xi\left( \frac{2n}{d} -d \right) =  \xi(2nd - d) = \xi(6d\ell + 3d) = (-1)^{\ell} \xi(3d).
\end{equation*}
By Lemma\,\ref{lem-xi} and\,\eqref{eq-a6=E}, we obtain
\begin{eqnarray*}
a_6(n) & = & (-1)^{\ell}  \sum_{d|n \, , \; d\, \mathrm{odd}} \, \xi(3d) \\
& = &  (-1)^{\ell+1} \, 2 E_1(n;4) = (-1)^{n+ \ell+1} \, a_2(n) /2.
\end{eqnarray*}
Since $n$ and $\ell$ are of the same parity, we have $a_6(n) = - a_2(n) /2$.

(c) Finally we consider the case when $n$ is divisible by~$3$. 
We write $n = 3^N t$, where $N\geq 1$ and $t$ is not divisible by~$3$.
Any odd divisor $d$ of~$n$ is of the form $d = 3^r s$ for some odd divisor~$s$ of~$t$ and $0 \leq r \leq N$.
Since $t$ and its divisors~$s$ are not divisible by~$3$ and since $s$ is odd, we again have $s \equiv 1,5,7$ or $11\!\pmod{12}$. 
Recall that for such~$s$ we have $s^2 \equiv 1 \!\!\pmod{12}$.
Thus, for $d = 3^r s$, we obtain
\begin{equation*}
\frac{2n}{d} - d \equiv \left( 2 \cdot 3^{N-r} t - 3^r \right) s \pmod{12}.
\end{equation*}

If $r=0$, then $2n/d-d \equiv (6\cdot 3^{N-1} t -1)s \pmod{12}$. Therefore,
\begin{equation*}
\xi\left( \frac{2n}{d} - d \right) = \xi\left((6\cdot 3^{N-1} t -1)s \right)  = (-1)^t \xi(-s) = (-1)^{t-1} \xi(s).
\end{equation*}
in view of\,\eqref{eq-sinus}.

If $0<r <N$, then $2n/d-d \equiv (6\cdot 3^{N-r-1} t - 3^r)s \pmod{12}$. Therefore,
\begin{equation*}
\xi\left( \frac{2n}{d} - d \right) = \xi\left((6\cdot 3^{N-r-1} t - 3^r)s \right)= (-1)^t \,\xi(-3^r s) = (-1)^{t-1} \,\xi(3^r s).
\end{equation*}
Now, $3^r \equiv 3 \hskip -3pt \pmod{12}$ if $r$ is odd, and $3^r \equiv -3$ if $r > 0$ is even. 
Then by\,\eqref{eq-sinus}, 
\begin{equation*}
\xi\left( \frac{2n}{d} - d \right) = (-1)^{t-r} \, \xi(3s).
\end{equation*}

Now consider the case $r = N$. If $N$ is odd, then $3^N \equiv 3 \hskip -3pt \pmod{12}$ and
\begin{equation*}
\xi\left( \frac{2n}{d} - d \right) =  \xi\left((2t-3^N)s \right) = \xi((2t-3)s) .
\end{equation*}

Now, if $t$ is odd, then $t \equiv 1,5,7$ or $11 \!\pmod{12}$.
We have $2t-3 \equiv 7$ or $11 \!\pmod{12}$ and the multiplication by~$7$ or by~$11$ exchanges
the sets $\{1,5\}$ and $\{7,11\}$. Since by\,\eqref{eq-xi} the function~$\xi$ takes opposite values on such sets, 
we have $\xi((2t-3)s) = -\xi(s)$.
Consequently, $\xi(2n/d-d) = -\xi(s)$ when $t$ is odd.

If $t$ is even, then $t \equiv 2,4,8$ or $10 \!\pmod{12}$. Then $2t-3 \equiv 1$ or $5 \!\pmod{12}$.
The multiplication by~$1$ or by~$5$ preserves each set $\{1,5\}$ and $\{7,11\}$, so that by\,\eqref{eq-xi}
we have $\xi((2t-3)s) = \xi(s)$.
In conclusion, 
\begin{equation*}
\xi(2n/d-d) = (-1)^t \, \xi(s)
\end{equation*}
when $r = N$ is odd.

If $r= N$ is even, then $3^N \equiv -3 \hskip -3pt \pmod{12}$ and
$\xi(2n/d-d) = \xi((2t-3^N)s) = \xi((2t+3)s)$. A reasoning as in the odd~$N$ case shows that when $N$ is even we have
\begin{equation*}
\xi(2n/d-d) = (-1)^{t-1} \, \xi(s).
\end{equation*}

We can now compute $a_6(n)$. We start with the case of odd $N$. Collecting the above information, we obtain
\begin{eqnarray*}
a_6(n) & = & \sum_{d|n \, , \; d\, \mathrm{odd}} \, \xi\left( \frac{2n}{d} - d\right) 
= \sum_{s|t \, , \; s\, \mathrm{odd}} \; \sum_{r=0}^N\, \xi\left( \frac{2\cdot 3^Nt}{3^r s} - d\right)\\
& = & \sum_{s|t \, , \; s\, \mathrm{odd}} \, \left( (-1)^{t-1} \xi(s) + \left(\sum_{r=1}^{N-1}\, (-1)^{t-r}\right)  \xi(3s) 
+ (-1)^t\, \xi(s) \right)\\
& = &  \left( (-1)^{t-1} + (-1)^t \right) \, \sum_{s|t \, , \; s\, \mathrm{odd}} \, \xi(s) = 0.
\end{eqnarray*}
On the other hand, since the power of~$3$ in~$n$ is odd, then by\,\eqref{eq-calculE} we have 
$a_2(n) = (-1)^n 4 \, E_1(n,4) = 0$.
Therefore, $a_6(n) = a_2(n)$ in this case.

If $N$ is even, then
\begin{eqnarray*}
a_6(n) & = & \sum_{d|n \, , \; d\, \mathrm{odd}} \, \xi\left( \frac{2n}{d} - d\right) 
= \sum_{s|t \, , \; s\, \mathrm{odd}} \;  \sum_{r=0}^N\, \xi\left(\frac{2\cdot 3^Nt}{3^r s} - d\right)\\
& = & \sum_{s|t \, , \; s\, \mathrm{odd}} \, \left( (-1)^{t-1} \xi(s) + \left(\sum_{r=1}^{N-1}\, (-1)^{t-r}\right)  \xi(3s) + (-1)^{t-1} \xi(s) \right)\\
& = & \sum_{s|t \, , \; s\, \mathrm{odd}} \, \left( 2 (-1)^{t-1}  \xi(s) + (-1)^{t-1} \xi(3s)  \right)
\\
& = &  (-1)^{t-1}\left( 2 \sum_{s|t \, , \; s\, \mathrm{odd}} \, \xi(s) 
+  \sum_{s|t \, , \; s\, \mathrm{odd}}  \, \xi(3s) \right) \\
& = & (-1)^t 4\, E_1(t;4)
\end{eqnarray*}
by Lemma\,\ref{lem-xi}.
Now, by multiplicativity of the excess fonction,
\[
E_1(n;4)= E_1(3^N;4) \, E_1(t;4) = E_1(t;4)
\]
since $E_1(3^N;4) = 1$ for even~$N$.  Finally, $t$ and $n$ being of the same parity, we have
\begin{equation*}
a_6(n) = (-1)^t \, 4\, E_1(t;4))  = (-1)^n \, 4\, E_1(n;4)) = a_2(n).
\end{equation*}
Q.e.d.
\end{proof}

\section{Proof of Theorem\,\ref{th-linear}}\label{sec-th-linear}

Set $f(q) = \eta(z)\eta(2z)\eta(3z)/\eta(6z) = \sum_{n\geq 0} \, a_6(n) \, q^n$ and 
$g(q) = \eta(z)^4/\eta(2z)^2 = \sum_{n\geq 0} \, a_2(n) \, q^n$; see\,\eqref{eq-a6eta}.
To prove Theorem\,\ref{th-linear} it suffices to check that
\[
f(q) = a g(q) + b g(jq) + c g(j^2q),
\]
where $a = 1/4$, $b = (1-j)/4$, and $c = (1-j^2)/4$.
Now,
\begin{eqnarray*}
a g(q) + b g(jq) + c g(j^2q)
& = & a \sum_{n\geq 0} \, a_2(n) \, q^n + b \sum_{n\geq 0} \, a_2(n) \, j^nq^n \\
&&+ c \sum_{n\geq 0} \, a_2(n) \, j^{2n}q^n \\
& = & (a+b+c) \sum_{m\geq 0} \, a_2(3m) \, q^{3m} \\
&& + (a+jb+j^2c) \sum_{m\geq 0} \, a_2(3m+1) \, q^{3m+1}\\
&&  + (a + j^2b + jc) \sum_{m\geq 0} \, a_2(3m+2) \, q^{3m+2}.
\end{eqnarray*}
It follows from\,\eqref{eq-a2a6} that
\begin{eqnarray*}
a g(q) + b g(jq) + c g(j^2q)
& = & (a+b+c) \sum_{m\geq 0} \, a_6(3m) \, q^{3m} \\
&& + 4(a+jb+j^2c) \sum_{m\geq 0} \, a_6(3m+1) \, q^{3m+1}\\
&&  - 2(a + j^2b + jc) \sum_{m\geq 0} \, a_6(3m+2) \, q^{3m+2}.
\end{eqnarray*}
The right-hand side is equal to~$f(q)$ since
$a+b+c = 1$, $a+jb+j^2c = 1/4$, and $a + j^2b + jc = -1/2$.
Q.e.d.

\section{An elementary proof of an identity by Victor Kac}\label{sec-Kac}

In\,\cite[p.~122]{Ka} Victor Kac derived four identities for $\eta$-products from his theory of contragredient Lie superalgebras.
One of these identities, labelled (new$_4$) in \emph{loc.~cit.}, can be rephrased in the following form:
\begin{equation*}
\frac{\eta^2(2z) \eta(3z)}{\eta(z) \eta(6z)} = \sum_{n\in \ZZ} \, (-1)^n f(n) \, q^{n^2},
\end{equation*}
where $f(3m) = 1$, $f(3m + 1) = -1$, and $f(3m + 2) = 0$.
This immediately implies
\begin{equation}\label{Kac-eta}
\frac{\eta^2(2z) \eta(3z)}{\eta(z) \eta(6z)} = \sum_{n\geq 0} \, b(n) \, q^{n^2},
\end{equation}
where $b(0) = 1$, $b(3m) = 2\, (-1)^m = 2\, (-1)^{3m}$ for $m>0$, and $b(n)  = (-1)^{n-1}$ if $n$ is not a multiple of~$3$. 
See also\,\cite[Th.\,8.2]{Ko}.

In a mail dated March 22, 2016 G\"unter K\"ohler observed that our $\eta$-product\,\eqref{def-eta6} 
can be written as the product of Kac's $\eta$-product\,\eqref{Kac-eta} and the $\eta$-product $\eta(z)^2/\eta(2z)$
(the latter two being modular forms of weight~$1/2$). Indeed,
\begin{equation}\label{obs-GK}
\frac{\eta(z)\eta(2z)\eta(3z)}{\eta(6z)} 
= \frac{\eta^2(2z) \eta(3z)}{\eta(z) \eta(6z)} \cdot \frac{\eta(z)^2}{\eta(2z)}.
\end{equation}
Now Gauss proved (see\,\cite[(7.324)]{Fi} or\,\cite[19.9\,(i)]{HW}) the following:
\begin{equation}\label{eq-Gauss}
\frac{\eta(z)^2}{\eta(2z)} = \sum_{n\in \ZZ} \, (-1)^n \, q^{n^2}
= \sum_{n\geq 0} \, a(n) \, q^{n^2},
\end{equation}
where $a(0) = 1$ and $a(n) = 2 (-1)^n$ for all $n>0$.
Therefore, the expansion of the right-hand side of~\eqref{obs-GK} is given by
\begin{equation*}
\frac{\eta^2(2z) \eta(3z)}{\eta(z) \eta(6z)} \cdot \frac{\eta(z)^2}{\eta(2z)}
=  \sum_{n \geq 0} \, a'_6(n) \, q^n,
\end{equation*}
where 
\begin{equation}\label{def-a'6}
a'_6(n) = \sum_{x, \, y \geq 0 \atop x^2 + y^2 = n} \, a(x) b(y).
\end{equation}
Consequently, an alternative way to prove Theorem\,\ref{th-a6r} is to establish the following lemma. 

\begin{lemma}\label{lem-GK}
For all $m\geq 0$,
\begin{eqnarray*}
a'_6(3m) & = & (-1)^m \, r(3m), \\
a'_6(3m+1) & = & (-1)^{m+1} \, \frac{r(3m+1)}{4}, \\
a'_6(3m+2) & = & (-1)^{m+1} \, \frac{r(3m+2)}{2}.
\end{eqnarray*}
\end{lemma}

Conversely, since the right-hand side of\,\eqref{eq-Gauss} is an invertible formal power series, 
the lemma combined with Gauss's identity\,\eqref{eq-Gauss} and with our elementary proof of Theorem\,\ref{th-a6r}
yields an elementary proof of Kac's identity\,\eqref{Kac-eta}.

\begin{proof}[Proof of Lemma\,\ref{lem-GK} (provided by G.~K\"ohler)]

(a) Suppose first $n = 3m +1$ for some integer $m \geq 0$.
We consider solutions $x\geq 0$, $y \geq 0$ of $x^2 + y^2 = n$. 
Since $n \equiv 1 \pmod{3}$, exactly one of the integers $x$, $y$ is a multiple of~$3$. 
Therefore the solutions can be coupled in pairs $((x,y), (y,x))$, where $3$ divides~$x$, but not~$y$ (hence~$y>0$).
If $x>0$, then the contribution of such a pair to~$a'_6(n)$ is 
\begin{multline*}
2 (-1)^x \cdot (-1)^{y-1} + 2 (-1)^y \cdot 2 (-1)^x\\ 
= 2 (-1)^{x+y} = 2 (-1)^{x^2+y^2} = 2(-1)^n = 2(-1)^{m+1}
\end{multline*}
and it is~$8$ for~$r(n)$, since one has to consider all pairs $(\pm x,\pm y)$ and $(\pm y,\pm x)$, which are distinct.
If~$x = 0$ (which occurs only if $n$ is a square), then the contribution is
\begin{equation*}
(-1)^{y-1} + 2 (-1)^y = (-1)^y = (-1)^{y^2} = (-1)^n = (-1)^{m+1}
\end{equation*}
for~$a'_6(n)$ and it is~$4$ for~$r(n)$ (corresponding to the pairs $(0,\pm y),(\pm y,0)$, which are distinct).
Summing up and comparing the contributions, we obtain $a'_6(n) = (-1)^{m+1} r(n)/4$,
which is the desired formula.

(b) Now let $n = 3m +2$ for some integer $m \geq 0$.
We again consider solutions $x\geq 0$, $y \geq 0$ of $x^2 + y^2 = n$. Since $n \equiv 2 \pmod{3}$,
none of $x$, $y$ is divisible by~$3$, and in particular $x>0$ and $y>0$.
The contribution of $(x,y)$ is
\begin{equation*}
2 (-1)^x \cdot (-1)^{y-1} = 2 (-1)^{x+y-1} = 2 (-1)^{x^2+y^2-1} = 2(-1)^{n-1} = 2(-1)^{m+1}
\end{equation*}
for~$a'_6(n)$ and it is~$4$ for~$r(n)$.
Summing up and comparing the contributions, we obtain $a'_6(n) = (-1)^{m+1} r(n)/2$.

(c) Finally let $n = 3m$. For $n=0$, the result is clear, so we may assume that $n>0$. 
Write $n = 3^N t$, where $N\geq 1$ and $t$ is not divisible by~$3$.
If~$N$ is odd, then by the remark preceding \eqref{eq-calculE} and by\,\eqref{eq-r=E} we have $r(n) = 4E(n;4) = 0$. 
Hence the sum\,\eqref{def-a'6} defining~$a'_6(n)$ is empty, which implies $a'_6(n) = 0 = (-1)^m r(3m)$. 

So let $N = 2s>0$ be even. It is easy to check that the solutions of $x^2 + y^2 = n = 3^{2s} t$ are of the form
$x = 3^s u$ and $y = 3^s v$, where $u^2 + v^2 = t$. 
If $t$ is not a square, then there is no solution where $u= 0$ or $v=0$, and we obtain
\begin{eqnarray*}
a'_6(n) 
& = & \sum_{u, \, v \geq 0 \atop u^2 + v^2 = t} \, 2(-1)^{3^s u} \cdot 2 (-1)^{3^{s}v}
= 4 \sum_{u, \, v \geq 0 \atop u^2 + v^2 = t} \, (-1)^u  (-1)^v \\
& = & 4 \sum_{u, \, v \geq 0 \atop u^2 + v^2 = t} \, (-1)^{u+v} 
= 4 \sum_{u, \, v \geq 0 \atop u^2 + v^2 = t} \, (-1)^{u^2+v^2}\\
& = & 4 (-1)^t \sum_{u, \, v \geq 0 \atop u^2 + v^2 = t} \, 1 \\
& = & (-1)^t r(t) = (-1)^m r(n).
\end{eqnarray*}
If $t = w^2$ is a square, then we have the additional solutions $(u,v) = (w,0)$ and $(u,v) = (0,w)$, 
hence $(x,y)=(3^sw,0)$ and $(x,y)=(0,3^sw)$.
This yields an additional contribution of 
$2(-1)^{3^sw}+2(-1)^{3^sw} = 4(-1)^{(3^sw)^2} = 4(-1)^n$ for~$a'_6(n)$, 
and for $r(n)$ it is~$4$; thus we have proved $a'_6(n) = (-1)^n r(n) = (-1)^m r(3m)$.
\end{proof}

\subsection*{Acknowledgment}

We are very grateful to G\"unter K\"ohler for having suggested the alternative proof of Theorem\,\ref{th-a6r} 
presented in Section\,\ref{sec-Kac}, provided the proof of Lemma\,\ref{lem-GK} and allowed us to include it in this note.
We also thank Robin Chapman, John McKay, Jean-Fran\c cois Mestre and Jean-Pierre Wintenberger for their remarks.


\end{document}